\documentclass[11pt]{article}%
\usepackage{amssymb}
\usepackage{amsmath}
\usepackage{euler}%
\usepackage{amsfonts}%
\usepackage{graphicx}
\providecommand{\U}[1]{\protect\rule{.1in}{.1in}}
\newtheorem{theorem}{Theorem}

\newtheorem{corollary}[theorem]{Corollary}

\newtheorem{lemma}[theorem]{Lemma}

\newtheorem{proposition}[theorem]{Proposition}

\newenvironment{proof}[1][Proof]{\noindent\textbf{#1.} }{\ \rule{0.5em}{0.5em}}
\begin{document}

\title{\textsf{Locatedness, Convexity, and Lebesgue Measurability---a New
Development}}
\author{\textsf{Douglas S. Bridges}}
\maketitle

\begin{abstract}%
\noindent
\textsf{Set within Bishop's framework for constructive analysis, this paper
contains a radically improved and expanded development of material, first
published in \cite{Br88}, dealing with the link between locatedness and
Lebesgue measurability for convex subsets of }$\mathbb{R}^{N}$.

\end{abstract}%

\normalfont\sf

\section{Introduction}%

\noindent
The framework of this paper\footnote{%
\normalfont\sf
\textbf{Keywords: \ }measurable, convex, located \ \textbf{MSC
classifications: \ }26E40, 03F60} is Bishop's constructive mathematics. We
assume some familiarity with, or at least access to, relevant parts of the
constructive integration theory found in Chapter 6 of \cite{BB}. Additional
sources for constructive analysis are \cite{Bishop}, \cite{BVtech}, and---the
most up-to-date reference---Ishihara's excellent Chapters 8 and 9 in
\cite{Handbook}.

Throughout the paper, $N$ will be a fixed positive integer, $\rho$ the
Euclidean metric on $\mathbb{R}^{N}$, and $\mu$ Lebesgue measure on
$\mathbb{R}^{N}$. We denote the Lebesgue integral in $\mathbb{R}%
,\mathbb{R}^{2},\ldots$ by $\int,\int\int,\ldots$. The
words\emph{\ integrable} and \emph{measurable} will refer exclusively to
Lebesgue measure theory.

A \textsc{complemented set} in $\mathbb{R}^{N}$ is an ordered pair$\ \left(
S,T\right)  $ of subsets of $\mathbb{R}^{N}$ such that\footnote{%
\normalfont\sf
In any normed space, $x\neq y$ means $\left\Vert x-y\right\Vert >0$. On the
real line $\mathbb{R}$, this notion of inequality is equivalent to the
negation of equality if and only if we accept Markov's principle, which is
widely regarded as of at least dubious constructive validity \cite[pp.
10--11]{BVtech}.} $x\neq y$ whenever $x\in S$ and $y\in A$. We use, for
example, boldface $\mathbf{A}$ to denote the complemented set $\mathbf{A}%
\equiv(A^{1},A^{0})$. We say that $x$ \textsc{belongs to} $\mathbf{A}$, and
write $x\in\mathbf{A}$, if $x\in A^{1}$. The \textsc{characteristic function}
of $\mathbf{A}$ is the mapping $\chi_{\mathbf{A}}$ with domain $A^{1}\cup
A^{0}$ such that $\chi_{\mathbf{A}}(x)=k$ if $x\in A^{k}$. If $\lambda>0$, we
define $\lambda\mathbf{A}\ $to be the complemented set $\left(  \lambda
A^{1},\lambda A^{0}\right)  $. If $\mathbf{B}\equiv(B^{1},B^{0})$ is a
complemented set such that $A^{1}\subset B^{1}$ and $B^{0}\subset A^{0}$, then
we write $\mathbf{A}<\mathbf{B}$ and say that $\mathbf{A}$ is a
\textsc{complemented subset} of $\mathbf{B}$; in that case, $\lambda
\mathbf{A}<\lambda\mathbf{B}$ for all $\lambda>0$. It is convenient to state
here a simple, but fundamental, result about complemented sets.

\begin{proposition}
\label{nov19p1}Let $K$ be a locally compact\footnote{%
\normalfont\sf
In constructive mathematics a l\textsc{ocally compact} metric space is one in
which each bounded set is included in a compact set. A mapping from a locally
compact metric space $X$ into a metric space is \textsc{continuous} if and
only if it is uniformly continuous on each compact subset of $X$.} subset of
$\mathbb{R}^{N}$, and $\mathbf{A}$ a complemented set in $\mathbb{R}^{N}$ such
that $K\subset A^{1}$. Then $\left(  K,-K\right)  <\mathbf{A}$ \cite[(6.6.6)]%
{BB}.
\end{proposition}

Important complemented sets in integration theory are of the following forms,
where $f$ is a real-valued function defined on a subset---$\mathsf{dmn\ }f$,
the \textsc{domain} of $f$---of $\mathbb{R}^{N}$:
\begin{align*}
\lbrack\![f\geqslant t]\!] &  \equiv\left(  \left\{  x\in\mathsf{dmn\ }%
f:f(x)\geqslant t\right\}  ,\left\{  x\in X:f(x)<t\right\}  \right)  ,\\
\lbrack\![f>t]\!] &  =\left(  \left\{  x\in\mathsf{dmn\ }f:f(x)>t\right\}
,\left\{  x\in X:f(x)\leqslant t\right\}  \right)  ,\\
\lbrack\![f\leqslant t]\!] &  =-[\![f>t]\!],\text{ and}\\
\lbrack\![f<t]\!] &  =-[\![f\geqslant t]\!].
\end{align*}
For each $x\in\mathbb{R}^{N}$ and each $r>0$, we denote by $B(x,r)$
(respectively, $\overline{B}(x,r)$) the open (respectively, closed) ball in
$\mathbb{R}^{N}$ with centre $x$ and radius $r$, and we define the
complemented sets%
\[
\mathbf{B}(x,r)\equiv\left(  B(x,r),-B(x,r)\right)  \text{ \ and \ }%
\overline{\mathbf{B}}(x,r)\equiv\left(  \overline{B}(x,r),-\overline
{B}(x,r)\right)  ,
\]
where, for any subset $S$ of $\mathbb{R}^{N}$,%
\[
-S\equiv\left\{  x\in\mathbb{R}^{N}:x\neq s\text{ for all }s\in S\right\}
\]
is the \textsc{metric complement }of $S$ in $\mathbb{R}^{N}$.

We say that a complemented set $\mathbf{A}\subset\mathbb{R}^{N}$

\begin{itemize}
\item[$\bullet$] is \textsc{bounded} if $A^{1}$ is a bounded subset of
$\mathbb{R}^{N}$;

\item[$\bullet$] is \textsc{integrable} (respectively, \textsc{measurable)} if
$\chi_{A}$ is an integrable (respectively, measurable) function;

\item[$\bullet$] \textsc{has positive (possibly infinite) measure} if it is
measurable and there exists an integrable complemented set $\mathbf{B}$ in
$\mathbb{R}^{N}$ such that $\mathbf{B}<\mathbf{A}$ and $\mu(\mathbf{B})>0$.
\end{itemize}

Note that $\mathbf{B}(x,r)$ and $\overline{\mathbf{B}}(x,r)$ are integrable,
their measure being the hypervolume of a ball of radius $r$ in $\mathbb{R}%
^{N}$. In view of this and \cite[(6.7.11)]{BB}, bounded measurable subsets of
$\mathbb{R}^{N}$ are integrable. We call a subset $F$ of $\mathbb{R}^{N}$ a
\textsc{full set} if there exists an integrable function $f$ such that
$\mathsf{dmn\,}f\subset F$. A countable intersection of full sets is full
\cite[(6.2.6)]{BB}.

Our objective is to present improved versions and proofs of \cite[Thms 1, 2,
and 4]{Br88}---namely:

\begin{theorem}
\label{aug03t1}Let $S$ be a located\thinspace\footnote{%
\normalfont\sf
In this paper, \emph{located }means \emph{located and inhabited}.} convex
subset of $\mathbb{R}^{N}$ with inhabited interior. Then the complemented set
$\mathbf{S}\equiv(S,-S)$ is Lebesgue measurable.\footnote{%
\normalfont\sf
In particular, this theorem applies when $S$ is a compact convex subset of
$\mathbb{R}^{N}$. In the constructive setting the statement \emph{every
compact subset of }$\mathbb{R}$\emph{\ is Lebesgue integrable }is essentially
nonconstructive: see page 63 of \cite{BR} (noting that in the relevant
paragraph, `$a_{n}=0$ for all $n$' should read `$a_{n}=1$ for all $n $'.}
\end{theorem}

\begin{theorem}
\label{jul21t1}Let $\mathbf{S}$ be a measurable complemented set in
$\mathbb{R}^{N}$ with positive (possibly infinite) measure, such that $S^{1}$
is convex. Then $S^{1}$ is located in $\mathbb{R}^{N}$.
\end{theorem}

\begin{theorem}
\label{jul21t2}Let $S$ be a convex subset of $\mathbb{R}^{N}$ such that $%
\mathord{\sim}%
S\ $is inhabited and the complemented set $\mathbf{S}\equiv(S,%
\mathord{\sim}%
S)$ is measurable. Then $%
\mathord{\sim}%
S$ is located in $\mathbb{R}^{N}$.
\end{theorem}%

\noindent
Here%
\[%
\mathord{\sim}%
S\equiv\left\{  x\in\mathbb{R}^{N}:x\neq s\text{ for all }s\in S\right\}
\]
is the \textsc{complement} of $S$ in $\mathbb{R}^{N}$.%

\smallskip

Although they are along the right lines, the proofs of the counterpart
theorems in \cite{Br88} are defective in at least two ways. First, the
identification of a subset $X$ of $\mathbb{R}^{N}$ with the complemented set
$(X,%
\mathord{\sim}%
X)$ at best could be confusing, especially when the correct identification
when $X$ is a compact set is actually with $\left(  X,-X\right)  $. Secondly,
more care needed to be taken with full sets: for example, the argument in the
proof of \cite[Thm.1]{Br88}---the counterpart of Theorem \ref{aug03t1}
above---implicitly assumed $S\cup-S$ is a full set, something that is not
needed in the proof of our counterpart. Such defects in the earlier proofs
are, we trust, properly dealt with in our work in the present paper.\footnote{%
\normalfont\sf
Some of the results in the paper---notably in the section on measure---are
simple ones that are included because their proofs are not easily, if all,
found in the literature.}

\section{Convexity}

We begin by gathering together a number of properties of convex sets. Several
of the proofs below improve on those found in \cite{Br88}; others are new.

\begin{lemma}
\label{nov10l2}Let $S$ be an inhabited subset of $\mathbb{R}^{N}$, and
$\lambda>0$.

\begin{itemize}
\item[\emph{(i)}] If $S$ is located, then so is $\lambda S$; in fact, for each
$x\in X$, $\rho(x,\lambda S)=\lambda\rho(\lambda^{-1}x,S)$.

\item[\emph{(ii)}] If $S$ is totally bounded (respectively, compact), then so
is $\lambda S$.

\item[\emph{(iii)}] If $S$ is convex, then so is $\lambda S$.

\item[\emph{(iv)}] If $S$ is convex and contains $0$, and $\lambda>1$, then
$S\subset\lambda S$.

\item[\emph{(v)}] If $S$ is located and convex and has inhabited metric
complement, then $-\lambda S$ is located and equals $\lambda(-S)$.
\end{itemize}
\end{lemma}

\begin{proof}
It is simple to prove (i)--(iv). Under the hypotheses of (v), we see from (i)
and (iii) that $\lambda S$ is located and convex; also, there exists
$y\in\mathbb{R}^{N}$ such that $\rho(\lambda y,\lambda S)=\lambda\rho(y,S)>0$
and therefore $y\in-\lambda S$. Hence, by \cite[4.1.15]{BVtech}, $-\lambda S$
is located. Moreover,%
\[
x\in-\lambda S\Leftrightarrow\lambda\rho(\lambda^{-1}x,S)>0\Leftrightarrow
\rho(\lambda^{-1}x,S)>0\Leftrightarrow\lambda^{-1}x\in-S\Leftrightarrow
x\in\lambda(-S),
\]
so $\lambda(-S)=-\lambda S$.
\end{proof}

\begin{lemma}
\label{0105a3}Let $\left\{  x_{1},\ldots,x_{N}\right\}  $ be a basis of
$\mathbb{R}^{N}$. Then $b\equiv\tfrac{1}{N+1}%
{\textstyle\sum_{k=1}^{N}}
x_{k}$ is in the interior of the convex hull of $\{0,x_{1},\ldots,x_{N}\}$.
\end{lemma}

\begin{proof}
We can unambiguously define a norm $\left\Vert \ \right\Vert ^{\prime}$ on
$\mathbb{R}^{N}$ by $\left\Vert \sum_{k=1}^{N}\lambda_{k}x_{k}\right\Vert
^{\prime}=\sum_{k=1}^{N}\left\vert \lambda_{k}\right\vert $. By the
equivalence of norms on finite-dimensional spaces, there exists $c>0$ such
that $\left\Vert \ \right\Vert ^{\prime}\leqslant c\left\Vert \ \right\Vert $,
where the latter norm is the original one on $X$. Let $x\equiv\sum_{k=1}%
^{N}\lambda_{k}x_{k}\in X$ satisfy%
\[
\ \left\Vert x-b\right\Vert <r\equiv\frac{1}{N(N+1)c}.
\]
Then%
\[
\sum_{k=1}^{N}\left\vert \lambda_{k}-\frac{1}{N+1}\right\vert =\left\Vert
x-b\right\Vert ^{\prime}<\frac{1}{N(N+1)},
\]
so for each $k$,%
\[
0<\frac{1}{N+1}-\frac{1}{N(N+1)}<\lambda_{k}<\frac{1}{N+1}+\frac{1}{N(N+1)}%
\]
and therefore%
\[
\sum_{k=1}^{N}\lambda_{k}<\frac{N}{N+1}+\frac{N}{N(N+1)}=1.
\]
Thus $x$ belongs to the convex hull of $\left\{  0,x_{1},\ldots,x_{N}\right\}
$. It follows that $B(b,r)$ is a subset of that convex hull.
\end{proof}%

\medskip

\begin{lemma}
\label{jul31l1}Let $S$ be a convex subset of $\mathbb{R}^{N}$, $a\in S$, and
$r>0$ be such that $B(a,r)\subset S$. Let $x\in S$, $0<t<1$, and $b=ta+\left(
1-t\right)  x$. Then $B(b,tr)\subset S$\emph{.}
\end{lemma}

\begin{proof}
Given $y\in B(b,tr)$, let
\[
z=\left(  1-\frac{1}{t}\right)  x+\frac{1}{t}y.
\]
Then $y-b=t(z-a)$ and therefore%
\[
\left\Vert z-a\right\Vert =\frac{1}{t}\left\Vert y-b\right\Vert <r\text{.}%
\]
Hence $z\in B(a,r)\subset S$, and therefore $y=tz+(1-t)x\in S$, by convexity.
Thus $B(b,tr)\subset S$.
\end{proof}%

\medskip

If $S$ is a located convex subset of $\mathbb{R}^{N}$ with inhabited metric
complement, then for each $\alpha>0$ we introduce the mapping%
\begin{equation}
\varphi_{\alpha}:x\rightsquigarrow\rho(x,-(1+\alpha)S)\text{ \ \ }%
(x\in\mathbb{R}^{N}).\label{ee}%
\end{equation}
This is a good definition, in view of Lemma \ref{nov10l2}(v). The function
$\varphi_{\alpha}$ is Lipschitz:%
\[
\left\vert \varphi_{\alpha}(x)-\varphi_{\alpha}(y)\right\vert \leqslant
\left\Vert x-y\right\Vert \ \ \ \left(  x,y\in\mathbb{R}^{N}\right)  .
\]

\begin{lemma}
\label{jul31l2}Let $S\subset\mathbb{R}^{N}$ be a totally bounded convex set
with interior point $0$, let $r>0$ be such that $B(0,r)\subset S$, and let
$0<\alpha<1$. Then
\[
d\equiv\inf\left\{  \rho(x,-(1+\alpha)S):x\in S\right\}
\]
exists, and $d\geqslant\alpha r$ \emph{\cite[Lemma 4]{Br88}}.
\end{lemma}

\begin{proof}
Since $\varphi_{\alpha}$ is uniformly continuous, its infimum $d$ over the
totally bounded set $S$ exists. Suppose that $d<\alpha r$. Then there exist
$x\in S$ and (by Lemma \ref{nov10l2}) $y\in-S$ such that%
\[
\alpha r>\left\Vert x-(1+\alpha)y\right\Vert =\left(  1+\alpha\right)
\left\Vert \frac{1}{1+\alpha}x-y\right\Vert
\]
and therefore%
\[
\left\Vert \frac{1}{1+\alpha}x-y\right\Vert <\frac{\alpha}{1+\alpha}r\text{.}%
\]
Taking $a=0$ and $t=\frac{\alpha}{1+\alpha}$ in Lemma \ref{jul31l1}, we see
that $B(\frac{1}{1+\alpha}x,\frac{\alpha}{1+\alpha}r)\subset S$ and therefore
$y\in S$, a contradiction from which we conclude that $d\geqslant\alpha r$.
\end{proof}

\begin{lemma}
\label{jul14l1}Let $S$ be a convex subset of a normed space $X$, let $\xi\in
S^{\circ}$, and let $r>0$ be such that $B(\xi,r)\subset S.$ Let $x\neq\xi,$
$0<t<1,$ and $z=t\xi+\left(  1-t\right)  x.$ If the ball $B(x,tr)$ intersects
$S,$ then $B(z,t^{2}r)\subset S\ $\cite[5.1.1]{BVtech}.
\end{lemma}%

\medskip

Let $S$ be a subset of a metric space $X$. We say that $S^{\circ}$
is\textbf{\ }\textsc{uniformly dense in} $S$ if for each $\varepsilon>0$ there
exists $\delta>0$ such that if $x\in S$, then $S\cap B(x,\varepsilon) $
contains a ball in $X$ of radius $\delta$. In that case, $S^{\circ}$ is
clearly dense in $S$.

\begin{lemma}
\label{0105a2}Let $S$ be a bounded convex set with inhabited interior in a
normed space. Then $S^{\circ}$ is uniformly dense in $S$.
\end{lemma}

\begin{proof}
Fix\textsf{\ }$\xi\in S$ and $r>0$ such that $B(\xi,r)\subset S$. There exists
$R>r$ such that $S\subset B(\xi,R)$. Given $\varepsilon$ with $0<\varepsilon
<\min\{1,R\}$, let $t=\varepsilon/2R$ \ and \ $\delta=\min\{\varepsilon
,t^{2}r\}$; then $0<t<1$. For each $x\in S$ either $\left\Vert x-\xi
\right\Vert <\frac{3r}{4}$ or $x\neq\xi$. In the first case, for each $y\in
B(x,\delta)$ we have $\left\Vert y-x\right\Vert <\varepsilon$ and%
\[
\left\Vert y-\xi\right\Vert \leqslant\left\Vert y-x\right\Vert +\left\Vert
x-\xi\right\Vert <t^{2}r+\frac{r}{4}<\left(  \frac{\varepsilon}{R}\right)
^{2}\frac{r}{4}+\frac{3r}{4}=r\,;
\]
whence%
\[
B(x,\delta)\subset B(\xi,r)\cap\overline{B}(x,\varepsilon)\subset
S\cap\overline{B}(x,\varepsilon).
\]
In the case $x\neq\xi$, write $z\equiv t\xi+(1-t)x$. Since $x\in S\cap
B(x,tr)$, we see from Lemma \ref{jul14l1} that $B(z,\delta)\subset S$. On the
other hand, if $z^{\prime}\in B(z,\delta)$, then%
\[
\left\Vert x-z^{\prime}\right\Vert \leqslant\left\Vert x-z\right\Vert
+\left\Vert z-z^{\prime}\right\Vert <t\left\Vert x-\xi\right\Vert
+\delta<tR+t^{2}r<tR+tR=\varepsilon.
\]
Hence $B(z,\delta)\subset S\cap B(x,\varepsilon)$. Thus in either case, there
is a ball of radius $\delta$ contained in $S\cap\overline{B}(x,\varepsilon)$.
\end{proof}

\section{Measure}

Turning now to measurability and integrability, by an $n$\textsc{-dimensional
rectangle\thinspace}\footnote{%
\normalfont\sf
https://en.wikipedia.org/wiki/Hyperrectangle} we mean the Cartesian product
$S\equiv I_{1}\times\cdots\times I_{n}$ of $n$ compact intervals in
$\mathbb{R}$. It readily follows from Fubini's theorem \cite[(6.9.7)]{BB} that
the complemented set $\left(  S,-S\right)  $ is then integrable, with measure
$\left\vert I_{1}\right\vert \times\cdots\times\left\vert I_{n}\right\vert $,
where $\left\vert I_{k}\right\vert $ is the length of $I_{k}$. This measure is
the hypervolume of the rectangle $S$.

\begin{lemma}
\label{jul03l1}Let $V$ be an $n$-dimensional subspace of $\mathbb{R}^{N}$,
where $1\leqslant n<N$. Then $\mathbf{V}\equiv\left(  V,-V\right)  $ is an
integrable complemented subset of $\mathbb{R}^{N}$, and $\mu(\mathbf{V})=0$.
\end{lemma}

\begin{proof}
Let $\left\{  e_{1},\ldots,e_{N}\right\}  $ be the standard orthonormal basis
of $\mathbb{R}^{N}$. We may assume that $V=\mathsf{span}\left\{  e_{1}%
,\ldots,e_{n}\right\}  $. For each positive integer $k$,%
\[
A_{k}\equiv\underset{n~\text{\textsf{terms}}}{\underbrace{\left[  -k,k\right]
\times\cdots\times\left[  -k,k\right]  }}\times\underset{N-n\text{
\textsf{terms}}}{\underbrace{\left\{  0\right\}  \times\cdots\times\{0\}}}%
\]
is a rectangle in $\mathbb{R}^{N}$, so, referring to the definition above, we
see that the complemented set $\mathbf{A}_{k}\equiv\left(  A_{k}%
,-A_{k}\right)  \ $is integrable, with measure $0$. Since $\sum_{k=1}^{\infty
}\mu(\mathbf{A}_{k})=0$, it follows from \cite[(6.3.10)]{BB} that
$\mathbf{A}\equiv%
{\displaystyle\bigvee\nolimits_{k\geqslant1}}
\mathbf{A}_{k}$ is integrable and $\mu\left(  \mathbf{A}\right)  =0$. It is
routine to verify that $A^{1}=V$ and $A^{0}=%
\mathord{\sim}%
V$. But $V$, being finite-dimensional, is complete and located, so, by
Bishop's Lemma (\cite[(4.3.8)]{BB}, \cite[3.1.1]{BVtech}), $%
\mathord{\sim}%
V=-V$ and therefore $\mathbf{A}=(V,-V)$.
\end{proof}

\begin{lemma}
\label{jul03l2}Let $\mathbf{A}$ be an integrable complemented set in
$\mathbb{R}^{N}$. Let $V$ be an $n$-dimensional subspace of $\mathbb{R}^{N}$,
where $1\leqslant n<N$, and let $\mathbf{V}$ be the complemented set $\left(
V,-V\right)  $. Then $\mathbf{V}\wedge\mathbf{A}$ is integrable, with measure
$0$.
\end{lemma}

\begin{proof}
By Lemma \ref{jul03l1}, $\mathbf{V}$ is integrable and $\mu(\mathbf{V})=0$; so
$\mathbf{V}\wedge\mathbf{A}$ is integrable \cite[(6.3.5)]{BB}, $\mu\left(
\mathbf{V}\wedge\mathbf{A}\right)  \leqslant\mu(\mathbf{V})$ \ \cite[(6.3.7)]%
{BB}, and therefore $\mu\left(  \mathbf{V}\wedge\mathbf{A}\right)  =0$.
\end{proof}

\begin{lemma}
\label{jul21l1}Let $\mathbf{A}$ and $\mathbf{B}$ be, respectively, a
measurable and an integrable complemented set in $\mathbb{R}^{N}$. Then
$\mathbf{A\wedge B}$ is integrable and $\mu(\mathbf{A\wedge B})\leqslant
\mu(\mathbf{B})$.
\end{lemma}

\begin{proof}
Since both $\chi_{\mathbf{B}}$ and (being integrable) $\chi_{\mathbf{A}}$ are
measurable, $\chi_{\mathbf{A}\wedge\mathbf{B}}=\chi_{\mathbf{A}}%
\chi_{\mathbf{B}}$ is also measurable \cite[(6.7.7)]{BB}. But $0\leqslant
\chi_{\mathbf{A}}\chi_{\mathbf{B}}\leqslant\chi_{\mathbf{B}}$ on the full set
$\left(  A^{1}\cup A^{0}\right)  \cup(B^{1}\cup B^{0})$, so $\chi
_{\mathbf{A\wedge B}}$ is integrable \cite[(6.7.11)]{BB}.
\end{proof}

\begin{lemma}
\label{nov18l1}Let $\mathbf{A},\mathbf{P},$ and $\mathbf{Q}$ be complemented
sets such that $\mathbf{P}$ and $\mathbf{Q}$ are integrable, $\mathbf{P}%
<\mathbf{A}<\mathbf{Q}$, and $\mu(\mathbf{P})=\mu(\mathbf{Q})$. Then
$\mathbf{A}$ is integrable, and $\mu(\mathbf{A})=\mu(\mathbf{P})$.
\end{lemma}

\begin{proof}
Since $\mu(\mathbf{Q-P})=0$, $(\mathbf{Q-P})^{0}$ is a full set \cite[(6.3.2)]%
{BB} as therefore is%
\[
F\equiv\left(  P^{1}\cup P^{0}\right)  \cap(\mathbf{Q-P})^{0}.
\]
Now, $P^{1}\subset A^{1}$ and%
\[
P^{0}\cap(\mathbf{Q-P})^{0}=P^{0}\cap(Q^{0}\cup P^{1})=P^{0}\cap Q^{0}%
=Q^{0}\subset A^{0},
\]
so $\chi_{\mathbf{A}}=\chi_{\mathbf{P}}$ on $F$. The result now follows from
\cite[(6.2.9)]{BB}.\label{000025}~
\end{proof}

\begin{lemma}
\label{nov18l2}Let $\mathbf{A}$ and $\mathbf{B}$ be complemented sets in
$\mathbb{R}^{N}$ such that for each $\varepsilon>0$ there exist integrable
sets $\mathbf{P}$ and $\mathbf{Q}$ with $\mathbf{P}<\mathbf{A}<\mathbf{B}%
<\mathbf{Q}$ and $\mu(\mathbf{Q})-\mu(\mathbf{P})<\varepsilon$. Then both
$\mathbf{A}$ and $\mathbf{B}$ are integrable, with $\mu(\mathbf{A}%
)=\mu(\mathbf{B})$.
\end{lemma}

\begin{proof}
Construct sequences $\left(  \mathbf{P}_{n}\right)  _{n\geqslant1}$ and
$\left(  \mathbf{Q}_{n}\right)  _{n\geqslant1}$ of integrable sets such that
for each $n$, $\mathbf{P}_{n}\subset\mathbf{P}_{n+1}\subset\mathbf{A}%
\subset\mathbf{B}\subset\mathbf{Q}_{n+1}\subset\mathbf{Q}_{n}$ and
$\mu(\mathbf{Q}_{n})-\mu(\mathbf{P}_{n})<2^{-n}$. Then $\ \left(
\mu(\mathbf{P}_{n})\right)  _{n\geqslant1}$ and $\left(  \mu(\mathbf{Q}%
_{n})\right)  _{n\geqslant1}$ are Cauchy sequences in $\mathbb{R}^{N}$ with a
common limit $\mu$. It follows from \cite[(6.3.9)]{BB} that $\mathbf{P}\equiv%
{\textstyle\bigvee_{n\geqslant1}}
\mathbf{P}_{n}$ and $\mathbf{Q\equiv}%
{\textstyle\bigwedge\limits_{n\geqslant1}}
\mathbf{Q}_{n}$ are integrable, with $\mu(\mathbf{P})=\mu(\mathbf{Q})$.
Moreover, $\mathbf{P}<\mathbf{A}<\mathbf{B}<\mathbf{Q}$, so it remains to
apply Lemma \ref{nov18l1}.
\end{proof}

\begin{lemma}
\label{jul02l1}Let $\mathbf{A}$ be a measurable complemented set in
$\mathbb{R}^{N}$ with positive (possibly infinite) measure. Then $A^{1}$
contains a basis of $\mathbb{R}^{N}$.
\end{lemma}

\begin{proof}
Suppose first that $\mathbf{A}$ is integrable, so that $\mu(\mathbf{A})>0$. By
\cite[(6.3.4)]{BB}, $A^{1}$ is inhabited by some point $a$. Since $\left(
\left\{  a\right\}  ,-\left\{  a\right\}  \right)  $ is integrable, with
measure $0$, we see from \cite[(6.3.5) and (6.3.6)]{BB} that $\mathbf{A}%
-\left(  \left\{  a\right\}  ,-\left\{  a\right\}  \right)  $ is integrable
and has measure%
\[
\mu(\mathbf{A})-\mu\left(  \left(  \left\{  a\right\}  ,-\left\{  a\right\}
\right)  \right)  =\mu(\mathbf{A})>0\text{.}%
\]
Thus, by \cite[(6.3.4)]{BB}, $A^{1}-\left\{  a\right\}  $ is inhabited by some
point $b$. Since $a\neq b$, it follows that $A^{1}$ contains a nonzero point
$e_{1}$ of $\mathbb{R}^{N}$. Now assume that for some $n<N$, we have
constructed linearly independent vectors $e_{1},\ldots,e_{n}$ belonging to
$A^{1}$. Let $V$ be the $n$-dimensional subspace of $\mathbb{R}^{N}$ spanned
by these vectors, and let $\mathbf{V}=\left(  V,-V\right)  $. By Lemma
\ref{jul03l2}, $\mathbf{V}\wedge\mathbf{A}$ is integrable and $\mu
(\mathbf{V}\wedge\mathbf{A})=0$. Hence $\mathbf{A}-\mathbf{V}\wedge\mathbf{A}$
is integrable and has measure%
\[
\mu(\mathbf{A}-\mathbf{V}\wedge\mathbf{A})=\mu(\mathbf{A})>0.
\]
Thus, again by \cite[(6.3.4)]{BB}, there exists $e_{n+1}\in A^{1}\cap-V$. Then
$\rho(e_{n+1},V)>0\,$, so the vectors $e_{1},\ldots,e_{n+1}$ in $A^{1} $ are
linearly independent \cite[4.1.10]{BVtech}. The desired conclusion for
integrable $\mathbf{A}$ now follows by induction. In the general case, there
is an integrable complemented set $\mathbf{B}<\mathbf{A}$ with $\mu
(\mathbf{B})>0$, so by the foregoing, $B^{1}$---and therefore $A^{1}%
$---contains a basis of $\mathbb{R}^{N}$.
\end{proof}

\begin{lemma}
\label{jul17l1}Let $\mathbf{S}$ be a measurable complemented set in
$\mathbb{R}^{N}$ with positive (possibly infinite) measure, such that $S^{1}$
is convex. Then the interior of $\ S^{1}$ is inhabited.
\end{lemma}

\begin{proof}
Translating if necessary, we may assume that\textbf{\ }$0\in S^{1}$. By Lemma
\ref{jul02l1}, $S^{1}$ contains a basis $\left\{  x_{1},\ldots,x_{N}\right\}
$ of $\mathbb{R}^{N}$. By Lemma \ref{0105a3}, $b\equiv\frac{1}{N+1}\sum
_{n=1}^{N}x_{n}$\ belongs to the interior of the convex hull of $\left\{
0,x_{1},\ldots,x_{N}\right\}  $ and hence to the interior of $S^{1}$.
\end{proof}%

\medskip

We now have an elementary lemma on integration for which there seems to be no
published reference.

\begin{lemma}
\label{aug17l1}Let $f\ $be an integrable function on $\mathbb{R}^{N}$, let
$\alpha>0$, and let $g$ be the mapping $x\rightsquigarrow\alpha^{-1}x$ of
$\mathbb{R}^{N}$ onto $\mathbb{R}^{N}$. Then $f\circ g$ is integrable and%
\[
\underset{N~\text{integrals}}{\int\int...\int}\left(  f\circ g\right)
=\alpha^{N}\int f.
\]

\end{lemma}

\begin{proof}
We illustrate the proof with the case $N=2$. First consider a continuous
function $f$ with compact support $K$. Since $g$ is a uniformly continuous
bijection of $\mathbb{R}$ onto $\mathbb{R}$ with uniformly continuous inverse,
$f\circ g$ is continuous and has compact support $g^{-1}(K)$. Thus, with the
change of variable\thinspace\footnote{%
\normalfont\sf
Since $f$ and $f\circ g$ are continuous functions with compact support, the
change-of-variable formula for \emph{Riemann }integration can be applied
here.} given by $x=\alpha u$ and $y=\alpha v$, we have%
\begin{align*}
\int\int f\circ g &  =\int\int_{g^{-1}(K)}f\circ g(u,v)\,\mathsf{d}%
u\mathsf{d}v=\int\int_{g^{-1}(K)}f(\alpha u,\alpha v)\,\mathsf{d}%
u\mathsf{d}v\\
&  =\alpha^{2}\int\int_{K}f(x,y)\text{\thinspace}\mathsf{d}x\mathsf{d}%
y=\alpha^{2}\int\int f.
\end{align*}
Turning now to the general case, let $\left(  f_{n}\right)  _{n\geqslant1}$ be
a representation of the integrable function $f$: in other words, each
$f_{n}:\mathbb{R}^{2}\rightarrow\mathbb{R}$ is a uniformly continuous function
with compact support, $\sum_{n=1}^{\infty}\int\int\left\vert f_{n}\right\vert
$ converges to $\int\int f$, and $f(x)=\sum_{n=1}^{\infty}f_{n}(x)$ for all
$x\ $in the full set $F$ on which $\sum_{n=1}^{\infty}\left\vert
f_{n}(x)\right\vert $ converges. By the foregoing, for each $n$ we have
$\int\int f_{n}\circ g=\alpha^{2}\int\int f_{n}$ and $\int\int\left\vert
f_{n}\circ g\right\vert =\alpha^{2}\int\int\left\vert f_{n}\right\vert $. Thus
$\sum_{n=1}^{\infty}\int\int\left\vert f_{n}\circ g\right\vert $ converges to
$\alpha^{2}\sum_{n=1}^{\infty}\int\int\left\vert f_{n}\right\vert $; whence,
by \cite[(6.2.15)]{BB}, $\sum_{n=1}^{\infty}\left(  f_{n}\circ g\right)  $
converges on a full set $G$, the function $\sum_{n=1}^{\infty}\left(
f_{n}\circ g\right)  $ with domain $G$ is integrable, and%
\[
\int\int\sum_{n=1}^{\infty}\left(  f_{n}\circ g\right)  =\sum_{n=1}^{\infty
}\int\int\left(  f_{n}\circ g\right)  =\alpha^{2}\sum_{n=1}^{\infty}\int\int
f_{n}=\alpha^{2}\int\int f.
\]
For each $x$ in the full set $F\cap G$ we have%
\[
\sum_{n=1}^{\infty}\left(  f_{n}\circ g\right)  (x)=\sum_{n=1}^{\infty}%
f_{n}(\alpha x)=f(\alpha x)=f\circ g(x).
\]
Thus $f\circ g$ is integrable, and $\int\int f\circ g=\alpha^{2}\int\int f $.
\end{proof}

\begin{corollary}
\label{aug17c1}Let $\mathbf{A}$ be an integrable complemented set in
$\mathbb{R}^{N}$, and let $\lambda>0$. Then $\lambda\mathbf{A}$ is integrable,
and $\mu(\lambda\mathbf{A})=\lambda^{N}\mu(\mathbf{A})$.
\end{corollary}

\begin{proof}
Let $g$ be the mapping $x\rightsquigarrow\lambda^{-1}x$ of $\mathbb{R}^{N}$
onto itself. By Lemma \ref{aug17l1}, the function $\chi_{\mathbf{A}}\circ g$
is integrable, and $\int\left(  \chi_{\mathbf{A}}\circ g\right)  =\lambda
^{N}\int\chi_{\mathbf{A}}$. Thus\textsf{\ }$\mathsf{dmn}(\chi_{\mathbf{A}%
}\circ g)$ is a full set, as therefore is%
\[
F\equiv\mathsf{dmn}(\chi_{\mathbf{A}}\circ g)\cap(A^{1}\cup A^{0})\text{.}%
\]
If $u\in F$, then%
\[
\chi_{\mathbf{A}}\circ g(u)=1\Leftrightarrow\lambda^{-1}u=g(u)\in
A^{1}\Leftrightarrow g(u)\in\lambda A^{1}=\left(  \lambda A\right)  ^{1}%
\]
and similarly, $\chi_{\mathbf{A}}\circ g(u)=0\Leftrightarrow g(u)\in\left(
\lambda A\right)  ^{0}$. Thus $\chi_{\mathbf{A}}\circ g=\chi_{\lambda
\mathbf{A}}$ on the full set $F$; whence, by \cite[(6.2.9)]{BB},
$\lambda\mathbf{A}$ is integrable and has measure $\lambda^{N}\mu(\mathbf{A})$.
\end{proof}%

\medskip

If $P(x)$ is a property of real numbers $x$ belonging to a subset $A~$of
$\mathbb{R}$, and $C$ is a countable subset of $A$ such that $P(x)$ holds for
all $x$ in $A\cap%
\mathord{\sim}%
C$, then we say that $P(x)$ holds \textsc{for all but countably many} $x$ in
$A$.

\begin{lemma}
\label{nov18l3}Let $S$ be a located subset of $\mathbb{R}^{N}$ containing $0$.
Then $S\cap\overline{B}(0,\alpha)$ is totally bounded for all but countably
many $\alpha>0$.
\end{lemma}

\begin{proof}
For each positive integer $n$ there exists a totally bounded subset $T_{n}$ of
$S$ such that $S\cap\overline{B}(0,n)\subset T_{n}$. For all but countably
many positive $\alpha<n$,%
\[
S\cap\overline{B}(0,\alpha)=\left\{  x\in T_{n}:\left\Vert x\right\Vert
\leqslant\alpha\right\}
\]
(which contains $0$) is totally bounded \cite[2.2.13]{BVtech}. Since the union
of countably many countable sets is countable, there exists a countable set
$C\subset(0,\infty)$ such that $S\cap\overline{B}(0,\alpha) $ is totally
bounded for all $\alpha\in\mathbb{R}\cap%
\mathord{\sim}%
C$.
\end{proof}%

\medskip

If $a,b$ are points in a normed linear space $V$, then%
\[
\left[  a,b\right]  \equiv\left\{  \left(  1-t\right)  a+tb:0\leqslant
t\leqslant1\right\}
\]
is called a \textsc{line segment}. The next lemma is well-known.

\begin{lemma}
\label{jul31l3}Let $a,b$ be points in a normed linear space $V$, and let
$p(x),q(x)$ be properties of points $x\in V$, such that $p(a)$ and $q(b)$.
Suppose that for each $x$ in the line segment $\left[  a,b\right]  $ either
$p(x)$ or $q(x)$. Then for each $\varepsilon>0$ there exist $x,y\in\left[
a,b\right]  $ such that $p(x),q(y)$, and $\left\Vert x-y\right\Vert
<\varepsilon$ \emph{\cite[Lemma 5]{Br88}}.
\end{lemma}

\begin{lemma}
\label{nov13l1}Let $f:\mathbb{R}^{N}\rightarrow\mathbb{R}$ be a continuous
function, and for each $r>0$ let $\mathbf{P}_{r}=[\![f\leqslant r]\!]$.
Suppose that there exists $c>0$ such that $P_{c}^{1}$ is bounded and
$P_{r}^{1}$ is inhabited for all $r$ with $0<r\leqslant c$. Then for all but
countably many $r$ with $0<r\leqslant c$, $P_{r}^{1}$ is compact\ and
$\mathbf{P}_{r}$ is integrable.
\end{lemma}

\begin{proof}
Choose $R>0$ such that $P_{c}^{1}\subset B\equiv\overline{B}(0,R)$, and define%
\[
P_{r}^{1}=\left\{  x\in B:f(x)\leqslant r\right\}  \ \ \ \ (0\leqslant
r\leqslant c).
\]
Since $f$ is uniformly continuous on $B$, $P_{r}^{1}$ is compact\ for all but
countably many $r$ with $0<r<c\ $\cite[(4.4.9)]{BB}. By \cite[(6.7.9)]{BB},
$f$ is measurable, since it is continuous and the function $x\rightsquigarrow
x$ is measurable. Letting $\mathbf{B}\equiv\left(  B,-B\right)  $, we see that
$\left\vert \chi_{\mathbf{B}}f\right\vert $, being measurable and bounded by
the integrable function $x\rightsquigarrow\sup_{x\in B}\left\vert
f(x)\right\vert \chi_{\mathbf{B}}(x)$,$\ $is integrable \cite[(6.7.11)]{BB}.
Thus for all but countably many $r>0$, $[\![\chi_{\mathbf{B}}f>r]\!]$ is
integrable \cite[(6.4.11)]{BB}. Fixing such $r$, we see that%
\[
\mathbf{A}_{r}\equiv\mathbf{B}\wedge\lbrack\![\chi_{\mathbf{B}}f\leqslant
r]\!]=\mathbf{B}-[\![\chi_{\mathbf{B}}f>r]\!]
\]
is integrable \cite[(6.3.5) and (6.3.6)]{BB}, so%
\[
F\equiv\mathsf{dmn\,}\mathbf{B}\cap\mathsf{dmn\,}[\![\chi_{\mathbf{B}}f>r]\!]
\]
is a full set. It is simple to show that $\chi_{\mathbf{A}_{r}}=\chi
_{\mathbf{P}_{r}}$ on $F$; whence $\mathbf{P}_{r}$ is integrable, with measure
$\mu(\mathbf{A}_{r})$ \cite[(6.2.9)]{BB}.
\end{proof}%

\medskip

The intersection of finitely many half-spaces in a normed linear space is
called a \textsc{convex polytope}\emph{.}%

\medskip

\begin{lemma}
\label{nov02l2}Let $S$ be a compact convex subset of $\mathbb{R}^{N}$ with $0
$ in its interior, and define complemented sets $\mathbf{S}\equiv(S,-S)$ and
$\mathbf{S}^{\circ}=(S^{\circ},-S)$. Let $0<\alpha<1$, and%
\[
d\equiv\inf\left\{  \rho(x,-(1+\alpha)S):x\in S\right\}  .
\]
Then there exists a uniformly continuous mapping $f:\mathbb{R}^{N}%
\rightarrow\mathbb{R}$ such that for all but countably many $r$ with
$0<r\leqslant\frac{d}{4}$ the following properties hold.

\begin{enumerate}
\item[\emph{(i)}] $P_{r}^{1}\equiv\left\{  x\in\mathbb{R}^{N}:f(x)\leqslant
r\right\}  $ is a compact convex polytope.

\item[\emph{(ii)}] The complemented set $\mathbf{P}_{r}\equiv\lbrack
\![f\leqslant r]\!]$ is integrable.

\item[\emph{(iii)}] $\mathbf{P}_{r}<(1+\alpha)\mathbf{S}^{\circ}%
<(1+\alpha)\mathbf{S}<(1+\alpha)\mathbf{P}_{r}$.
\end{enumerate}
\end{lemma}

\begin{proof}
By Lemma \ref{nov10l2}, $S\ $is a subset of $(1+\alpha)S$, which is compact
and convex. Define $\varphi_{\alpha}$ as at \ref{ee}. Then%

\medskip

\emph{If }$\varphi_{\alpha}(x)>0$\emph{, then} $x\in--(1+\alpha)S=(1+\alpha
)S^{\circ}$,%
\hfill
(\dag)%

\medskip
%

\noindent
the second `$=\,$' coming from \cite[Thm 22]{dsb26}. By Lemma \ref{jul31l2},
$d=\inf_{x\in S}\varphi_{\alpha}(x)\geqslant\alpha r$. Noting that, in view of
\cite[5.1.5]{BVtech},
\[
\inf\left\{  \varphi_{\alpha}(x):x\in(1+\alpha)S\right\}  =0
\]
and using \cite[2.2.14]{BVtech}, we can find $\delta$ with $0<\delta<\frac
{d}{4}$ such that%
\[
T\equiv\left\{  x\in(1+\alpha)S:\varphi_{\alpha}(x)\leqslant\delta\right\}
\]
is compact. Let $\left\{  x_{1},\ldots,x_{m}\right\}  $ be a $\frac{d}{4}%
$-approximation to $T$. For each $k\leqslant m$ and each $s\in S$,%
\[
\left\Vert x_{k}-s\right\Vert \geqslant\varphi_{\alpha}(s)-\varphi_{\alpha
}(x_{k})\geqslant d-\frac{d}{4}=\frac{3d}{4}\text{.}%
\]
Hence $\rho(x_{k},S)\geqslant\tfrac{3d}{4}$. By the separation theorem
(\cite[(6.4.3)]{BB},\cite[5.2.9]{BVtech}), there exists a normable linear
functional $u_{k}$ on $\mathbb{R}^{N}$ such that $\left\Vert u\right\Vert =1$
and $u_{k}(x_{k})>u_{k}(x)+\frac{d}{2}$ for each $x\in S$. Since $S$ is
compact and each $u_{k}$ is uniformly continuous,%
\[
\sigma_{k}\equiv\sup\left\{  u_{k}(x):x\in S\right\}
\]
exists and $u_{k}(x_{k})\geqslant\sigma_{k}+\frac{d}{2}$. Consider the
uniformly continuous mapping $f:\mathbb{R}^{N}\rightarrow\mathbb{R}$ defined
by%
\[
f(x)\equiv\max_{1\leqslant k\leqslant m}\left(  u_{k}(x)-\sigma_{k}\right)  .
\]
For each $r$ with $0<r\leqslant\frac{d}{4}$ define%
\[
P_{r}^{1}\equiv%
{\textstyle\bigcap_{k=1}^{m}}
\left\{  x\in\mathbb{R}^{N}:u_{k}(x)\leqslant\sigma_{k}+r\right\}  =\left\{
x\in\mathbb{R}^{N}:f(x)\leqslant r\right\}  ,
\]
which is a convex polytope in $\mathbb{R}^{N}$. For each $x\in S$ and each
$k\leqslant m$,%
\[
u_{k}(x)-\sigma_{k}\leqslant u_{k}(x)-u_{k}(x_{k})<-\frac{d}{2};
\]
so%
\begin{equation}
f(x)\leqslant-\frac{d}{2}<r\ \ \ \ \ (x\in S).\label{uu}%
\end{equation}
Hence $S\subset P_{r}^{1}$ and therefore $\mathbf{S}^{\circ}<\mathbf{S}%
<\mathbf{P}_{r}$, the second `$=$' following from Proposition \ref{nov19p1},
and therefore%
\begin{equation}
(1+\alpha)\mathbf{S}^{\circ}<(1+\alpha)\mathbf{S}<(1+\alpha)\mathbf{P}%
_{r}.\label{kk}%
\end{equation}
We now prove that%
\begin{equation}
P_{r}^{1}\subset(1+\alpha)S^{\circ}.\label{qq}%
\end{equation}
Given $x\in P_{r}^{1}$, suppose that $\varphi_{\alpha}(x)<\frac{\delta}{2} $.
Since $0\in S^{\circ}\subset(1+\alpha)S$ and therefore%
\[
\varphi_{\alpha}(0)=\rho(0,-(1+\alpha)S)>0,
\]
we can apply Lemma \ref{jul31l3} to construct points $a_{1},a_{2}$ in the line
segment $\left[  0,x\right]  $ such that $\left\Vert a_{1}-a_{2}\right\Vert
<\frac{\delta}{4}$, $\varphi_{\alpha}(a_{1})>0$, and $\varphi_{\alpha}%
(a_{2})<\frac{\delta}{2}$. Then $a_{1}\in(1+\alpha)S^{\circ}$, by (\dag).
Also,%
\[
\varphi_{\alpha}(a_{1})\leqslant\left\Vert a_{1}-a_{2}\right\Vert
+\varphi_{\alpha}(a_{2})<\frac{\delta}{4}+\frac{\delta}{2}<\delta,
\]
so $a_{1}\in T$. Hence there exists $k$ such that $\left\Vert a_{1}%
-x_{k}\right\Vert <\frac{d}{4}$ and therefore%
\begin{align*}
u_{k}(a_{1}) &  \geqslant u_{k}(x_{k})-\left\Vert u_{k}\right\Vert \left\Vert
a_{1}-x_{k}\right\Vert \\
&  >\sigma_{k}+\frac{d}{2}-\frac{d}{4}=\sigma_{k}+\frac{d}{4}>\sigma_{k}+r.
\end{align*}
But $0\in S\subset P_{r}^{1}$, so $[0,x]$ is a subset of the convex set
$P_{r}^{1}$; whence $a_{1}\in P_{r}^{1}$ and therefore $u_{k}(a_{1}%
)\leqslant\sigma_{k}+r$, a contradiction from which we conclude that
$\varphi_{\alpha}(x)>0$ and therefore, by (\dag), that $x\in\left(
1+\alpha\right)  S^{\circ}$. This completes the proof of (\ref{qq}). Taking
$r=\frac{d}{4}$ in (\ref{qq}), we see that $\left\Vert x\right\Vert
\leqslant(1+\alpha)\sup_{s\in S}\left\Vert s\right\Vert $ for all $x\in
P_{d/4}^{1}$. Hence, by Lemma \ref{nov13l1}, for all but countably many $r$
with $0<r\leqslant\frac{d}{4}$, $P_{r}^{1}$ is compact and $\mathbf{P}_{r}%
\ $is integrable. For such $r$ we see from Proposition \ref{nov19p1} that
$\mathbf{P}_{r}<(1+\alpha)\mathbf{S}^{\circ}$. Taken in conjunction with
(\ref{kk}), this proves (iii).
\end{proof}%

\medskip

A sequence $(f_{n})_{n\geqslant1}$ of measurable functions on $\mathbb{R}^{N}$
\textsc{converges in measure} to a fully defined real function $f$ if for each
integrable complemented set $\mathbf{A}$ and each $\varepsilon>0$ there exist
a positive integer $N$ such that for each $n\geqslant N$ there exists an
integrable complemented set $\mathbf{B}$ with $B^{1}\subset A^{1}$,
$\mu(\mathbf{A}-\mathbf{B})<\varepsilon$, and $\left\vert f(x)-f_{n}%
(x)\right\vert <\varepsilon$ for all $x\in B^{1}$ \cite[(6.8.1)]{BB}.

Observe that if $\left(  \alpha_{n}\right)  _{n\geqslant1}$ a strictly
increasing sequence of positive numbers tending to\ infinity, and
$\mathbf{K}_{n}\equiv\mathbf{B}(0,\alpha_{n})\subset\mathbb{R}^{N}%
$\ \ $\left(  n\geqslant1\right)  $, then $(\mathbf{K}_{n})_{n\geqslant1}$ is
a \textsc{Lebesgue basis} of $\mathbb{R}^{N}\,$---that is, each $\mathbf{K}%
_{n}$ is integrable, $K_{1}^{1}\subset K_{2}^{1}\subset\cdots$, $%
{\textstyle\bigcup\limits_{n\geqslant1}}
K_{n}^{1}$ is a full set, and $\left(  \chi_{K_{n}}\right)  _{n\geqslant1}$
converges in measure to $1$.

\begin{proposition}
\label{aug26p1}Let $f$ be a nonnegative function defined on a subset of
$\mathbb{R}^{N}$, $\left(  \alpha_{n}\right)  _{n\geqslant1}$ a strictly
increasing sequence of positive numbers tending to $\infty$, and
$\mathbf{K}_{n}=\mathbf{B}(0,\alpha_{n})\ $for each $n$. Suppose that
$f\chi_{\mathbf{K}_{n}}$ is integrable for each $n$. Then $f$ is measurable.
\end{proposition}

\begin{proof}
Let $D=\mathsf{dmn\,}f\subset\mathbb{R}^{N}$. Since $f\chi_{\mathbf{K}_{n}} $
is integrable, $D\cap\left(  K_{n}^{1}\cup K_{n}^{0}\right)  =\mathsf{dmn\,}%
f\chi_{\mathbf{K}_{n}}$ is a full set. Hence
\[%
{\textstyle\bigcap\limits_{n\geqslant1}}
(D\cap\left(  K_{n}^{1}\cup K_{n}^{0}\right)  )=D\cap%
{\textstyle\bigcap\limits_{n\geqslant1}}
\left(  K_{n}^{1}\cup K_{n}^{0}\right)  ,
\]
and therefore $D$ itself, is full. Let $\mathbf{A}$ be an integrable
complemented set, and for each $n$ let%
\[
\mathbf{B}_{n}=\left(  \left(  D\cap A^{1}\cap K_{n}^{1}\right)  ,\left(
D\cap(A^{0}\cup K_{n}^{0}\right)  \right)  .
\]
Then $B_{n}^{1}\subset A^{1}$, and $\chi_{\mathbf{B}_{n}}=\chi_{\mathbf{A}%
\wedge\mathbf{K}_{n}}$ on%
\[
G\equiv D\cap((A_{n}^{1}\cap K_{n}^{1})\cup(A_{n}^{0}\cup K_{n}^{0}),
\]
which is full since $D$ is full and $\mathbf{A\wedge K}_{n}$ is integrable.
Thus, by \cite[(6.2.9)]{BB}, $\mathbf{B}_{n}$ is integrable and $\mu
(\mathbf{B}_{n})=\mu(\mathbf{A}\wedge\mathbf{K}_{n})$. Now, by \cite[(6.8.10)]%
{BB}, $\mu(\mathbf{A}\wedge\mathbf{K}_{n})\rightarrow\mu(\mathbf{A)}$ as
$n\rightarrow\infty$. Given $\varepsilon>0$, choose $N$ such that%
\[
0\leqslant\mu(\mathbf{A}-\mathbf{B}_{n})=\mu(\mathbf{A})-\mu(\mathbf{B}%
_{n})=\mu(\mathbf{A})-\mu(\mathbf{A\wedge K}_{n})<\varepsilon
\]
for all $n\geqslant N$. For such $n$, if $x\in B_{n}^{1}=D\cap A^{1}\cap
K_{n}^{1}$, then $x\in D$ and $\chi_{\mathbf{K}_{n}(x)}=1$, so $\left(
f\chi_{\mathbf{K}_{n}}\right)  (x)=f(x)$ and therefore $\left\vert
f(x)-f\chi_{\mathbf{K}_{n}}(x)\right\vert <\varepsilon$. Since $\mathbf{A}$
and $\varepsilon$ are arbitrary, we conclude that the sequence $\left(
f\chi_{n}\right)  _{n\geqslant1}$ of integrable, and therefore measurable,
functions converges in measure. Hence $f$ is measurable \cite[(6.8.2)]{BB}.%

\hfill

\end{proof}

\begin{corollary}
\label{nov24c1}Let $\mathbf{S}$ be a complemented set in $\mathbb{R}^{N}$,
$\left(  \alpha_{n}\right)  _{n\geqslant1}$ a strictly increasing sequence of
positive numbers tending to $\infty$, and $\mathbf{K}_{n}=\mathbf{B}%
(0,\alpha_{n})\ $for each $n$. Suppose that $\mathbf{S}\wedge\mathbf{K}_{n} $
is integrable for each $n$. Then $\mathbf{S}$ is measurable.
\end{corollary}%

\medskip

It is tempting to seek a constructive proof of this classically true proposition:

\begin{quote}
\emph{If }$\left(  \mathbf{S}_{n}\right)  _{n\geqslant1}$\emph{\ is an
increasing sequence of Lebesgue integrable complemented sets in }%
$\mathbb{R}^{N}$ \emph{such that for each }$n$\emph{, }$S_{n}^{1}$\emph{\ is
compact, then }$\mathbf{S}\equiv%
{\textstyle\bigvee_{n\geqslant1}}
\mathbf{S}_{n}$\emph{\ is Lebesgue measurable.}%
\hfill
(*)
\end{quote}%

\noindent
Alas, such a search cannot succeed, even if we also require that each $S_{n}$
be convex.%

\medskip

\begin{proof}
[Brouwerian Example 1]Given an increasing binary sequence $\left(
a_{n}\right)  _{n\geqslant1}$, define $S_{n}=\left[  0,1\right]  $ if
$a_{n}=0$, $S_{n}=\left[  0,2\right]  $ if $a_{n}=1$, and $\mathbf{S}%
_{n}=\left(  S_{n},-S_{n}\right)  $. Then each $S_{n}$ is convex, totally
bounded, and integrable. Suppose that $\mathbf{S}\equiv%
{\textstyle\bigvee_{n\geqslant1}}
\mathbf{S}_{n}$ is measurable. Since $\chi_{\mathbf{S}}\leqslant\chi_{\left(
\left[  0,2\right]  ,\mathbb{R}-\left[  0,2\right]  \right)  }$ on the full
set $%
{\textstyle\bigcap\limits_{n\geqslant1}}
(S_{n}^{1}\cup S_{n}^{0})$, we see from \cite[(6.7.11)]{BB} that the
complemented set $\mathbf{S}$ is integrable. Either $\mu(\mathbf{S})>1$ or
$\mu(\mathbf{S})<2$. In the first case, $\mu(\mathbf{S-}\left(  \mathbb{R}%
^{N}-\left[  0,1\right]  \right)  )>0,$ so, by \cite[Ch. 6, (3.4)]{BB}, there
exists $x\in S^{1}-\left[  0,1\right]  $. Choosing $N$ such that $x\in S_{N}$,
we see that $a_{N}=1$. On the other hand, if $\mu\mathbf{(S})<2$, then we
cannot have $a_{n}=1$, so $a_{n}=0$ for all $n$. Thus (*) implies that for
each increasing binary sequence $\left(  a_{n}\right)  _{n\geqslant1}$ either
$a_{n}=0$ for all $n$ or else there exists $m$ such that $a_{m}=1$. This
statement is easily seen to be equivalent to the omniscience
principle\footnote{%
\normalfont\sf
See pages 8--9 of \cite{BVtech}.} \textsf{LPO}, so (**) is essentially nonconstructive.
\end{proof}

\section{Proofs of the Main Theorems}

Before the proof of our first major theorem we have one simple lemma.%

\medskip

\begin{lemma}
\label{nov25l1}Let$~A$ and $B$ be inhabited subsets of a metric
space$\ (X,\rho)$ such that $B$ is open. Then $B\cap-(A\cap B)\subset-A$.
\end{lemma}

\begin{proof}
Since $B$ and $-(A\cap B)$ are open sets, for each $x\in B\cap-(A\cap B)$
there exists $r>0$ such that ~$B(x,r)\subset B\cap-(A\cap B)$. Suppose there
exists $y\in A$ with $\rho(x,y)<r$. Then $y\in A\cap B$ and $y\in-(A\cap B)$,
which is absurd. Hence $\rho(x,y)\geqslant r$ for all $y\in A$, and therefore
$x\in-A$.
\end{proof}%

\medskip

This brings us to the%

\medskip

\begin{proof}
[Proof of Theorem \ref{aug03t1}]Translating if necessary, we may assume that
$0\in S^{\circ}$. We first take the case where $S$ is bounded and therefore,
being located, totally bounded. We use the notation of Lemma \ref{nov02l2} and
its proof, \emph{except that where we write }$S$ \emph{therein, we here write}
$\overline{S}$. Choose $R>0$ such that $2\overline{S}\subset\overline{B}%
(0,R)$; then $2\overline{\mathbf{S}}<\mathbf{B}\equiv\overline{\mathbf{B}%
}(0,R)$, by Proposition \ref{nov19p1}. Given $\varepsilon>0$, choose $\alpha$
with $0<\alpha<1$ such that $\mu(\mathbf{B})<\frac{1+\alpha}{\alpha
}\varepsilon$. By Lemma \ref{nov02l2}, there exists $r$ such that
$\mathbf{P}_{r}$ is compact and integrable, and (since, by \cite[Propn.
23]{dsb26}, $\overline{S}^{\circ}=S^{\circ}$)%
\[
\mathbf{P}_{r}<(1+\alpha)\mathbf{S}^{\circ}<(1+\alpha)\overline{\mathbf{S}%
}<(1+\alpha)\mathbf{P}_{r}.
\]
Then%
\[
(1+\alpha)^{-1}\mathbf{P}_{r}<\mathbf{S}^{\circ}<\overline{\mathbf{S}%
}<\mathbf{P}_{r}<(1+\alpha)\overline{\mathbf{S}}<\mathbf{B.}%
\]
Hence%
\begin{align*}
0 &  \leqslant\mu(\mathbf{P}_{r})-\mu((1+\alpha)^{-1}\mathbf{P}_{r})\\
&  =(1-(1+\alpha)^{-1})\mu(\mathbf{P}_{r})\\
&  =\frac{\alpha}{1+\alpha}\mu(\mathbf{P}_{r})\leqslant\frac{\alpha}{1+\alpha
}\mu(\mathbf{B})<\varepsilon.
\end{align*}
Since $\varepsilon>0$ is arbitrary, it follows from Lemma \ref{nov18l2} that
$\mathbf{S}^{\circ}$ and $\overline{\mathbf{S}}$ are integrable and have the
same measure, and $S^{\circ}\cup-S$ is a full set. Since $\chi_{\mathbf{S}%
^{\circ}}=\chi_{\mathbf{S}}$ on that full set, $\mathbf{S}$ is integrable,
with $\mu(\mathbf{S})=\mu(\mathbf{S}^{\circ})=\mu(\overline{\mathbf{S}})$
\cite[(6.2.9)]{BB}.

We now consider the general case, in which $S$ is locally totally bounded but
not necessarily bounded. Using Lemma \ref{nov18l3}, construct a strictly
increasing sequence $(\alpha_{n})_{n\geqslant1}$ tending to infinity such that
$S\cap B(0,\alpha_{n})$ is totally bounded, and hence located, for each $n$.
Then $S\cap B(0,\alpha_{n})$ is convex and has interior point $0$, so by the
first part of the proof,
\[
\mathbf{S}_{n}\equiv\left(  S\cap B(0,\alpha_{n}),-(S\cap B(0,\alpha
_{n}))\right)
\]
is integrable. Since $\mathbf{B}(0,\alpha_{n})$ is integrable, it follows
that
\[
F\equiv\left[  \left(  S\cap B(0,\alpha_{n})\right)  \cup\,-(S\cap
B(0,\alpha_{n}))\right]  \cap\left(  B(0,\alpha_{n})\cup\,-B(0,\alpha
_{n})\right)
\]
is a full set. Consider $x\in F\cap-(S\cap B(0,\alpha_{n}))$. Either $x\in
B(0,\alpha_{n})$ and therefore, by Lemma \ref{nov25l1}, $x\in-S$, or else
$x\in-\overline{B}(0,\alpha_{n})$. In either case, $x\in-S\cup-B(0,\alpha
_{n})$, so%
\[
\chi_{\mathbf{S}_{n}}(x)=0=\chi_{\mathbf{S}}(x)\chi_{\mathbf{B}(0,\alpha_{n}%
)}(x)=\chi_{\mathbf{S}\wedge\mathbf{B}(0,\alpha_{n})}(x).
\]
Since also, $S_{n}^{1}=\left(  \mathbf{S}\wedge\mathbf{B}(0,\alpha_{n}\right)
)^{1}$, it follows that $\chi_{\mathbf{S}_{n}}=\chi_{\mathbf{S}\wedge
\mathbf{B}(0,\alpha_{n})}$ on $F$. Hence $\mathbf{S}\wedge\mathbf{B}%
(0,\alpha_{n})$ is integrable \cite[(6.2.9)]{BB}. Applying Corollary
\ref{nov24c1}, we now see that $\mathbf{S}$ is measurable.
\end{proof}%

\medskip

We denote the boundary of a set $S\subset\mathbb{R}^{N}$ by $\partial S$, and
we write $\mathbf{\partial S}\equiv\left(  \partial S,-\partial S\right)  $.

\begin{corollary}
\label{nov18c1}Let $S$ be a located convex subset of $\mathbb{R}^{N}$ with
inhabited interior. Then $\partial\mathbf{S}$ is integrable, with measure $0 $.
\end{corollary}

\begin{proof}
By Theorem \ref{aug03t1}, $\mathbf{S}^{\circ}$ is integrable, so $S^{\circ
}\cup-S$ is a full set. But $\partial\mathbf{S}=(\partial S,S^{\circ}\cup-S)$,
so the desired conclusion follows from \cite[(6.3.3)]{BB}.
\end{proof}%

\medskip

Next we have the%

\medskip

\begin{proof}
[Proof of Theorem \ref{jul21t1}]By Lemma \ref{jul17l1}, there exist $s_{0}\in
S^{1}$ and $r>0$ such that $B(s_{0},r)\subset S^{1}$. To begin with, assume
that\textbf{\ }$S^{1}$ is bounded, and let $B$ be a compact ball with
$S^{1}\subset B$. Then $\chi_{\mathbf{S}}\leqslant\chi_{(B,-B)}$ on the full
set $\left(  S^{1}\cup S^{0}\right)  \cap\left(  B,-B\right)  $, so
$\mathbf{S}$ is integrable \cite[(6.7.11)]{BB}; moreover, $\mu(\mathbf{S}%
)\geqslant\mu(\mathbf{B}(s_{0},r))>0$. By Lemma \ref{0105a2}, $\left(
S^{1}\right)  ^{\circ}$ is uniformly dense in $S^{1}$. Given $\varepsilon>0$,
pick $\delta>0$ such that if $x\in S^{1}$, then $S^{1}\cap B(x,\varepsilon)$
contains a ball in $\mathbb{R}^{N}$ of radius $\delta$. Applying
\cite[(6.6.7)]{BB}, construct a compact set $K$ such that $\mathbf{K}%
\equiv(K,-K)$ is integrable, $\mathbf{K}<\mathbf{S}$, and $\mu(\mathbf{S}%
-\mathbf{K})<\mu(\mathbf{B}(0,\delta))$. Let $F$ be a finite $\varepsilon
$-approximation to $K$, and consider any $x\in S^{1}$. There exists $z\in
S^{1}$ such that $B(z,\delta)\subset S^{1}\cap B(x,\varepsilon)$. If
$\rho\left(  x,K\right)  >\varepsilon$, then $B(x,\varepsilon)\subset-K$, so
$B(z,\delta)\subset S^{1}-K$, $\chi_{\mathbf{B}(z,\delta)}\leqslant
\chi_{\mathbf{S}-\mathbf{K}}$, and therefore $\mu\left(  \mathbf{S}%
-\mathbf{K}\right)  \geqslant\mu(\mathbf{B}(z,\delta))$, a contradiction.
Hence $\rho\left(  x,K\right)  <2\varepsilon$ and therefore $\rho
(x,F)<3\varepsilon$. Thus $F$ is a finite $3\varepsilon$-approximation to
$S^{1}$. Since $\varepsilon$ is arbitrary, it follows that $S^{1}$ is totally
bounded and hence located.

Now consider the general case, in which $S^{1}$ need not be bounded. Given
$a\in\mathbb{R}^{N}$, choose%
\[
R>4\left\Vert a-s_{0}\right\Vert +2\left\Vert s_{0}\right\Vert
\]
such that $\overline{B}(s_{0},r)\subset B(0,R)$. Then $\mathbf{A}%
\equiv\mathbf{S}\wedge\mathbf{B}(0,R)$ is integrable, by Lemma \ref{jul21l1}.
Also, $\overline{B}(s_{0},r)\subset S^{1}\cap B(0,R)=A^{1}$, so $\mu
(\mathbf{A})\geqslant\mu(\mathbf{B}(s_{0},r))>0$ \cite[(6.3.7)]{BB}. Since
$A^{1}$ is convex and bounded, it follows from the first paragraph of this
proof that $A^{1}$ is located in $\mathbb{R}^{N}$. Consider any $x\in S^{1}$.
Either $\left\Vert x\right\Vert <R$, in which case $x\in A^{1}$ and
$\left\Vert a-x\right\Vert \geqslant\rho(a,A^{1})$, or else $\left\Vert
x\right\Vert >R/2$. In the latter case, noting that $s_{0}\in A^{1}$, we have%
\begin{align*}
\left\Vert a-x\right\Vert  &  \geqslant\left\Vert x\right\Vert -\left\Vert
a-s_{0}\right\Vert -\left\Vert s_{0}\right\Vert \\
&  >\frac{R}{2}-\left\Vert a-s_{0}\right\Vert -\left\Vert s_{0}\right\Vert
>\left\Vert a-s_{0}\right\Vert \geqslant\rho(a,A^{1})\text{.}%
\end{align*}
Since $A^{1}\subset S^{1}$, it follows that $\rho(a,S^{1})$ exists and equals
$\rho(a,A^{1})$. But $a\in\mathbb{R}^{N}$ was arbitrarily chosen. Hence
$S^{1}$ is located.
\end{proof}

\begin{corollary}
\label{jul17c1}Let $\mathbf{S}$ be a measurable complemented set with positive
measure in $\mathbb{R}^{N}$ such that $S^{1}$ is bounded and convex, and $%
\mathord{\sim}%
S^{1}$ is inhabited. Then $%
\mathord{\sim}%
S^{1}$ is located in $\mathbb{R}^{N}$.
\end{corollary}

\begin{proof}
By Lemma \ref{jul17l1}, $\left(  S^{1}\right)  ^{\circ}$ is inhabited, so,
$S^{1}$ being convex, $\left(  S^{1}\right)  ^{\circ}$ is dense in $S^{1}$
\cite[Propn 10]{dsb26}. Hence $-S^{1}=-\left(  S^{1}\right)  ^{\circ}$.
Applying Lemma 5.1.4 of \cite{BVtech} to the inhabited, open, convex set
$\left(  S^{1}\right)  ^{\circ}$, we see that $-S^{1}$ is dense in $%
\mathord{\sim}%
\left(  S^{1}\right)  ^{\circ}$. Since $-S^{1}\subset%
\mathord{\sim}%
S^{1}\subset%
\mathord{\sim}%
\left(  S^{1}\right)  ^{\circ}$, it follows that $-S^{1}$ is dense in $%
\mathord{\sim}%
S^{1}$. Referring to Theorem \ref{jul21t1}, we see that $S^{1}$ is a located
convex subset of $\mathbb{R}^{N}$ such that $-S^{1}$ is inhabited; whence, by
\cite[4.1.15]{BVtech}, $-S^{1}$ is located in $\mathbb{R}^{N}$. Since $-S^{1}$
is dense in $%
\mathord{\sim}%
S^{1}$, the latter is located also.
\end{proof}%

\medskip

In the case $N=1$ we can remove from Theorem \ref{jul21t1} the hypothesis that
$\mathbf{S}$ have positive (possibly infinite) measure. Here is a simpler
proof than the one given in \cite[Thm. 3]{Br88}.

\begin{proposition}
\label{jul21p1}Let $\mathbf{S}$ be a measurable complemented set in
$\mathbb{R}$ such that $S^{1}$ is inhabited and convex. Then $S^{1}$ is located.
\end{proposition}

\begin{proof}
It will suffice to consider the case where $S^{1}$ is bounded, in which case
$\mathbf{S}$ is integrable, by Lemma \ref{jul21l1}: for we can then deal with
the general case by a simpler variant of the argument in the second paragraph
of the proof of Theorem \ref{jul21t1}. Pick $s\in S^{1}$. Given $\varepsilon
>0$, we have either $\mu(\mathbf{S})>0$ or $\mu(\mathbf{S})<\varepsilon/2$. In
the first case, by Theorem \ref{jul21t1}, $S^{1}$ is located in $\mathbb{R}$
and hence totally bounded. In the second case, if $x\in S^{1}$ and $\left\vert
x-s\right\vert >\varepsilon/2$, then since $S^{1}$ contains all points of the
interval between $s$ and $x$, we have the contradiction $\mu(\mathbf{S}%
)>\varepsilon/2$; whence $\left\vert x-s\right\vert <\varepsilon$ for all
$x\in S^{1}$ and therefore $\left\{  s\right\}  $ is an $\varepsilon
$-approximation to $S^{1}$. Thus for each $\varepsilon>0$ there exists an
$\varepsilon$-approximation to $S^{1}$, which is therefore totally bounded and
hence located in $\mathbb{R}$.
\end{proof}%

\medskip

We cannot drop the positive (possibly infinite) measure hypothesis in the case
$N>1$ even if we know that $S^{1}$ is inhabited.\footnote{%
\normalfont\sf
This example is an improvement on the one on page 419 of \cite{Br88}.}%

\medskip

\begin{proof}
[Brouwerian Example 2]Let $\left(  a_{n}\right)  _{n\geqslant1}$ be a binary
sequence with at most one term equal to $0$. Construct compact convex subsets
$S_{n}$ of $\mathbb{R}^{2}$ such that for each $n$

\begin{itemize}
\item[$\bullet$] if $a_{k}=0$ for all $k\leqslant n$, then $S_{n}=\left[
0,1\right]  \times\left\{  0\right\}  $;

\item[$\bullet$] if $a_{n}=1-a_{n-1}$ and $n$ is odd, then $S_{n}=\left[
0,2\right]  \times\left\{  0\right\}  $ for all $k\geqslant n$; and

\item[$\bullet$] if $a_{n}=1-a_{n-1}$ and $n$ is even, then for all
$k\geqslant n$, $S_{n}$ is the closed triangle with vertices $\left(
0,0\right)  ,\left(  1,0\right)  ,\left(  1,\frac{2}{n}\right)  $.
\end{itemize}%

\noindent
Then $\mathbf{S}_{n}\equiv\left(  S_{n},-S_{n}\right)  $ is integrable for
each $n$. Let $\mathbf{S}\equiv%
{\textstyle\bigvee_{n\geqslant1}}
\mathbf{S}_{n}$, which is inhabited and convex. To prove that $\mathbf{S}$ is
integrable, for each $k$ let $\mathbf{A}_{k}$ be the integrable complemented
set $%
{\textstyle\bigvee_{n=1}^{k}}
\mathbf{S}_{n}$; then $A_{1}^{1}\subset A_{2}^{1}\subset\cdots$. If $m>n$,
then one of three alternatives holds: $a_{k}=0$ for all $k\leqslant m$;
$a_{k}=1$ for some $k\leqslant n$; there exists $\nu$ such that $n<\nu
\leqslant m$ and $a_{\nu}=1$. In each of the first two cases, $\mu
(\mathbf{A}_{m})=\mu(\mathbf{A}_{n})$. In the third case, if $\nu$ is odd,
then $\mu(\mathbf{A}_{m})=\mu(\mathbf{A}_{n})=0$; if $\nu$ is even, then
$\mu(\mathbf{A}_{m})=\frac{1}{\nu}$. Thus in all cases, $\left\vert
\mu(\mathbf{A}_{m})-\mu(\mathbf{A}_{n})\right\vert \leqslant\frac{1}{n}$.
Hence $\left(  \mu(\mathbf{A}_{n})\right)  _{n\geqslant1}$ is a Cauchy
sequence and therefore converges to a limit in $\mathbb{R}$. It follows from
\cite[(6.3.8)]{BB} that $\mathbf{S}$ is integrable. On the other hand, if
$S^{1}$ is located, then either $\rho\left(  \left(  2,0\right)
,S^{1}\right)  >0$ or $\rho\left(  \left(  2,0\right)  ,S^{1}\right)  <1$. In
the first case, if $a_{m}=1-a_{m-1}$, then $m$ cannot be odd; in the second,
if $a_{m}=1-a_{m-1}$, then $m$ cannot be even. We now see that if Theorem
\ref{jul21t1} holds without the hypothesis $\mu(\mathbf{S})>0$, then we can
prove the omniscience principle\footnote{%
\normalfont\sf
See page 9 of \cite{BVtech}.} \textsf{LLPO: }for each binary sequence $\left(
a_{n}\right)  _{n\geqslant1}$ with at most one term equal to $1$, either
$a_{n}=0$ for all odd $n$ or $a_{n}=0$ for all even $n.$
\end{proof}%

\medskip

We now turn to Theorem \ref{jul21t2}, which strengthens Corollary
\ref{jul17c1}. We shall need a simple variant of the constructive
least-upper-bound principle \cite[(2.4.3)]{BB}\cite[2.1.18]{BVtech}---namely,
that if $A$ is an inhabited subset of $\mathbb{R}$ that is bounded below by a
real number $r$, such that for all real $\beta>r$ either $\inf A$ exists or
else there exists $x\in A$ with $x<\beta$, then $\inf A$ exists.%

\medskip

\begin{proof}
[Proof of Theorem \ref{jul21t2}]Since $S\cup%
\mathord{\sim}%
S$, being a full set, is dense in $\mathbb{R}^{N}$ \cite[(6.3.4)]{BB}, it will
suffice to prove that $\rho(x,%
\mathord{\sim}%
S)$ exists for each $x$ in $S$. Given such $x$, from Lemma \ref{jul21l1} we
see that $\mathbf{S\wedge B}(x,r)$ is integrable for all $r>0$. For each
positive integer $n$ let $v_{n}=\mu(\mathbf{B}(0,n^{-1}))\ $and construct an
increasing binary sequence $\left(  \lambda_{n}\right)  _{n\geqslant1}$ such
that for each $n$,%
\begin{align*}
\lambda_{n}=0 &  \Rightarrow\mu(\mathbf{S}\wedge\mathbf{B}\left(
x,n^{-1}\right)  )<v_{n},\\
\lambda_{n}=1 &  \Rightarrow\mu(\mathbf{S}\wedge\mathbf{B}\left(
x,n^{-1}\right)  )>0.
\end{align*}
In view of Corollary \ref{jul17c1},\label{25} we may assume that $\lambda
_{1}=0$. If $\lambda_{n}=0$, then by \cite[(6.3.6)]{BB},%
\[
\mu(\mathbf{B}\left(  x,n^{-1}\right)  -\mathbf{S})=v_{n}-\mu(\mathbf{S}%
\wedge\mathbf{B}\left(  x,n^{-1}\right)  )>0,
\]
so, by \cite[(6.3.4)]{BB}, we can choose $x_{n}\in S^{0}=%
\mathord{\sim}%
S$ with $\left\Vert x_{n}-x\right\Vert <1/n$. If $\lambda_{n}=1$, we set
$x_{n}=x_{n-1}$. If $m\geqslant n$, then%
\[
\left\Vert x_{m}-x_{n}\right\Vert \leqslant\left\Vert x_{m}-x\right\Vert
+\left\Vert x_{n}-x\right\Vert <\frac{1}{m}+\frac{1}{n}\leqslant\frac{2}{n}.
\]
Thus $\left(  x_{n}\right)  _{n\geqslant1}$ is a Cauchy sequence and therefore
converges to a limit $x_{\infty}$ in the closure of $%
\mathord{\sim}%
S$. Consider any $\beta$ with $0<\beta$. Either $\left\Vert x_{\infty
}-x\right\Vert >\alpha$ or $\left\Vert x_{\infty}-x\right\Vert <\beta$. In the
first case, there exists $m$ such that $\left\Vert x_{m}-x\right\Vert >1/m$;
then $\lambda_{m}=1$, $\mu(\mathbf{S})\geqslant$ $\mu(\mathbf{S}%
\wedge\mathbf{B}\left(  x,m^{-1}\right)  )>0$, and by Corollary \ref{jul17c1},
$%
\mathord{\sim}%
S$ is located. In the second case, there exists $n$ such that $\left\Vert
x_{n}-x\right\Vert <\beta$, where $x_{n}\in%
\mathord{\sim}%
S$. It follows from the remark immediately preceding this proof that $\rho(x,%
\mathord{\sim}%
S)$ exists.
\end{proof}%

\medskip

Theorem \ref{jul21t2} tells us that if, in Theorem \ref{jul21t1}, we replace
the hypothesis of positive (possibly infinite) measure by `$%
\mathord{\sim}%
S^{1}$ is inhabited', then we can prove that $%
\mathord{\sim}%
S^{1}$ is located even if, as Brouwerian Example 1 shows, we may not know
whether $S^{1}$ itself is located.%

\bigskip
%

\bigskip

%

\bigskip
%

\bigskip
%

\bigskip
%

\noindent
\textbf{{\small Author's address:}}{\small \ Department of Mathematics \&
Statistics, University of Canterbury, Christchurch 8140, New Zealand}%

\noindent
\textbf{{\small Author's email:} \ \thinspace}%
\texttt{{\small dugbridges@gmail.com}}

\bigskip

\vfill

\begin{flushright}
{\small \texttt{dsbridges 011226}}
\end{flushright}

\end{document}